\documentclass[reqno,11pt]{amsart}
\theoremstyle{plain}
\newtheorem{thm}{Theorem}[section]
\newtheorem{lemma}[thm]{Lemma} 
\newtheorem{prop}[thm]{Proposition}

\theoremstyle{remark}

\theoremstyle{definition}
\newtheorem{defi}[thm]{Definition}
\newtheorem{example}[thm]{Example}

\newtheorem{remark}[thm]{Remark}

\usepackage{amscd,amssymb,comment,euscript,graphics}
\usepackage{pict2e}
\usepackage{enumitem}
\usepackage[initials]{amsrefs}
\usepackage{color}

\newcount\theTime
\newcount\theHour
\newcount\theMinute
\newcount\theMinuteTens
\newcount\theScratch
\theTime=\number\time
\theHour=\theTime
\divide\theHour by 60
\theScratch=\theHour
\multiply\theScratch by 60
\theMinute=\theTime
\advance\theMinute by -\theScratch
\theMinuteTens=\theMinute
\divide\theMinuteTens by 10
\theScratch=\theMinuteTens
\multiply\theScratch by 10
\advance\theMinute by -\theScratch

\def\today{{\number\day\space
 \ifcase\month\or
  January\or February\or March\or April\or May\or June\or
  July\or August\or September\or October\or November\or December\fi
 \space\number\year}}

\newcommand\Ac{{\mathcal{A}}}
\newcommand\Cpx{{\mathbb C}}
\newcommand\Ec{{\mathcal{E}}}
\newcommand\eps{\epsilon}
\newcommand\Fc{{\mathcal{F}}}
\newcommand\ft{{\tilde f}}

\newcommand\Lc{{\mathcal{L}}}
\newcommand\Mcal{{\mathcal{M}}}
\newcommand\mut{{\tilde{\mu}}}
\newcommand\Proj{{\mathrm{Proj}}}
\newcommand\Reals{{\mathbb R}}
\newcommand\restrict{{\upharpoonright}}
\newcommand\Sc{{\mathcal{S}}}
\newcommand\tdet{{\textstyle\det}}

\begin{document}

\title[Determinants on operator bimodules]{Determinants associated to traces on operator bimodules}

\author[Dykema]{K. Dykema$^*$}
\address{K.\ Dykema, Department of Mathematics, Texas A\&M University, College Station, TX, USA.}
\email{ken.dykema@math.tamu.edu}
\thanks{\footnotesize ${}^{*}$ Research supported in part by NSF grant DMS--1202660.}
\author[Sukochev]{F. Sukochev$^{\S}$}
\address{F.\ Sukochev, School of Mathematics and Statistics, University of New South Wales, Kensington, NSW, Australia.}
\email{f.sukochev@math.unsw.edu.au}
\thanks{\footnotesize ${}^{\S}$ Research supported by ARC}
\author[Zanin]{D. Zanin$^{\S}$}
\address{D.\ Zanin, School of Mathematics and Statistics, University of New South Wales, Kensington, NSW, Australia.}
\email{d.zanin@math.unsw.edu.au}

\subjclass[2010]{46L52}
\date{September 23, 2016}
\keywords{Determinant, von Neumann algebra, II$_1$-factor, noncommutative function space}

\begin{abstract}
Given a II$_1$-factor $\Mcal$ with tracial state $\tau$ and given an $\Mcal$-bimodule
$\Ec(\Mcal,\tau)$ of operators affiliated to $\Mcal$
we show that traces on $\Ec(\Mcal,\tau)$
(namely, linear functionals that are invariant under unitary conjugation)
are in bijective correspondence with rearrangement-invariant linear functionals
on the corresponding symmetric function space $E$.
We also show that, given a positive
trace $\varphi$ on $\Ec(\Mcal,\tau)$,
the map $\tdet_\varphi:\Ec_{\log}(\Mcal,\tau)\to[0,\infty)$ defined by
$\tdet_\varphi(T)=\exp(\varphi(\log |T|))$ when $\log|T|\in\Ec(\Mcal,\tau)$ and $0$ otherwise,
is multiplicative on the $*$-algebra $\Ec_{\log}(\Mcal,\tau)$
that consists of all affiliated operators $T$ such that $\log_+(|T|)\in\Ec(\Mcal,\tau)$.
Finally, we show that all multiplicative maps on the invertible elements of $\Ec_{\log}(\Mcal,\tau)$
arise in this fashion.
\end{abstract}

\maketitle



\section{Introduction}

Let $\Mcal$ be a von Neumann algebra factor of type II$_1$, with tracial state $\tau$.
Assume $\Mcal$ has separable predual.
The Fuglede--Kadison determinant~\cite{FK52},
is the multiplicative map $\Delta_\tau:\Mcal\to[0,\infty)$
defined by
\begin{equation}\label{eq:FK}
\Delta_\tau(T)=\lim_{\eps\to0^+}\exp(\tau(\log(|T|+\eps)).
\end{equation}
In this paper, we prove
multiplicativity
of analogous determinants corresponding to arbitrary positive
traces on arbitrary $\Mcal$-bimodules
of affiliated operators.

Choose any normal representation of $\Mcal$ on a Hilbert space and let
$\Sc(\Mcal,\tau)$ be the $*$-algebra of (possibly unbounded) operators on the Hilbert space affiliated to $\Mcal$.
This algebra, often called the Murray-von Neumann algebra of $\Mcal$, is independent of the representation.
See, for example, Section~6 of~\cite{KL14} for an exposition of this theory.
Let $\Proj(\Mcal)$ denote the set of projections (i.e., self-adjoint idempotents) in $\Mcal$.
For $A\in\Sc(\Mcal,\tau)$ and $t\in(0,1)$,
$\mu(t,A)$ denotes the generalized singular number of $A$, defined by
\[
\mu(t,A)=\inf\{\|A(1-p)\|\mid p\in\Proj(\Mcal),\,\tau(p)\le t\},
\]
where $\|\cdot\|$ is the operator norm.
This goes back to Murray and von Neumann; see, for example, Section~2.3 of~\cite{LSZ13} for some basic theory.
We will write simply $\mu(A)$ for the function $t\mapsto\mu(t,A)$,
which is nonincreasing and right continuous.

Let $E$ be a complex vector space of measurable functions on $[0,1]$ with the property that if $f$ and $g$ are
measurable functions with $f^*\le g^*$
and $g\in E$, then $f\in E$, where $f^*$ denotes the decreasing rearrangement of $|f|$.
Following~\cite{LSZ13}, we will call such a space $E$ a Calkin function space.
Note that $f\in E$ implies that the dilation $D_2f$ lies in $E$, where $D_2f(t)=f(t/2)$.
In particular, every nonzero Calkin function space contains $L_\infty[0,1]$.
The corresponding $\Mcal$-bimodule $\Ec(\Mcal,\tau)$ is the set of all $A\in\Sc(\Mcal,\tau)$ such that $\mu(A)\in E$.
This correspondence, sometimes called the Calkin correspondence in the setting of $(\Mcal,\tau)$,
is a bijection from the set of all Calkin function spaces onto the set of all
operator $\Mcal$-bimodules, by which we mean
subspaces of $\Sc(\Mcal,\tau)$ that are closed under left and right multiplication by elements of $\Mcal$,
and it goes back to Guido and Isola~\cite{GI95}.
See Theorem 2.4.4 of~\cite{LSZ13} for the formulation used here.
An equivalent version of this is also described in~\cite{DK05}.
Note that if $\Ac\subseteq\Mcal$ is any unital abelian von Neumann subalgebra that is diffuse (i.e., has no minimal projections),
then the $*$-algebra
$\Sc(\Ac,\tau\restrict_\Ac)$ of affiliated operators
is naturally embedded in $\Sc(\Mcal,\tau)$ and, upon identifying $\Ac$ with $L_\infty(0,1)$,
the elements of $\Sc(\Ac,\tau\restrict_\Ac)$ are naturally identified with measurable
functions
on $(0,1)$.
Under these identifications, we have $E=\Sc(\Ac,\tau\restrict_\Ac)\cap\Ec(\Mcal,\tau)$.

By a trace on $\Ec(\Mcal,\tau)$, we mean
a linear functional $\varphi$ of $\Ec(\Mcal,\tau)$ such that $\varphi(UAU^*)=\varphi(A)$ for every $A\in\Ec(\Mcal,\tau)$
and every unitary $U\in\Mcal$.
A functional $\varphi_0$ of $E$ is said to be rearrangement-invariant if $\varphi_0(f)=\varphi_0(g)$ whenver $f,g\in E$, $f,g\ge0$
and $f^*=g^*$.

The difficult half of the following result is essentially proved in~\cite{KS08}.
The proof of the other half is similar to the proof of Lemma~9.4 of~\cite{DSZ.unbd}.
\begin{thm}\label{thm:trace}
Let $\Mcal$ be a II$_1$-factor with separable predual.
Let $E$ be a Calkin function space and let $\Ec(\Mcal,\tau)$ be the corresponding $\Mcal$-bimodule.
There is a bijection from the set of all traces of $\Ec(\Mcal,\tau)$ onto the set of all
rearrangement-invariant functionals of $E$,
whereby a trace $\varphi$ of $\Ec(\Mcal,\tau)$ is mapped to a functional $\varphi_0$ of $E$ satisfying
\begin{equation}\label{eq:varphidef}
\varphi_0(\mu(A))=\varphi(A)\text{ whenever }A\in\Ec(\Mcal,\tau)\text{ and }A\ge0.
\end{equation}
\end{thm}
\begin{proof}
Suppose $\varphi_0:E\to\Cpx$ is a rearrangement-invariant linear functional.
By the proof of (part of) Theorem 5.2 of~\cite{KS08}, there is a trace $\varphi:\Ec(\Mcal,\tau)\to\Cpx$ satisfying~\eqref{eq:varphidef}.
The statement of that theorem includes additional assumptions about $E$, namely, that it carries a rearrangement-invariant
complete norm.
However, the proof found in~\cite{KS08} is valid, verbatim, in the more general situation considered here.

Suppose $\varphi:\Ec(\Mcal,\tau)\to\Cpx$ is a trace.
We will now show that for any $A\in\Ec(\Mcal,\tau)$ that is positive, $\varphi(A)$ depends only on $\mu(A)$.
Indeed, let $A_1,A_2\in\Ec(\Mcal,\tau)$ be such that $A_1,A_2\ge0$ and $\mu(A_1)=\mu(A_2).$ Set 
$$B_k=\sum_{n\geq0}n1_{[n,n+1)}(A_k),\quad C_k=A_k-B_k,\quad k=1,2.$$
Clearly, positive operators $B_1$ and $B_2$ have discrete spectrum and $\mu(B_1)=\mu(B_2).$ Since $\Mcal$ is a factor, one can choose a unitary element $U\in\Mcal$ such that $B_1=UB_2U^{-1}.$ Clearly, $\varphi(B_1)=\varphi(UB_2U^{-1})=\varphi(B_2).$ By Theorem 2.3 in \cite{FH80}, we have
$\varphi\restrict_{\Mcal}=c_\varphi\tau\restrict_{\Mcal}$ for a constant $c_\varphi.$ For bounded positive operators $C_1$ and $C_2,$ we have $\mu(C_1)=\mu(C_2)$ and also,
therefore,
$$\varphi(C_1)=c_\varphi\tau(C_1)=c_\varphi\tau(C_2)=\varphi(C_2).$$
Thus, we get
$$\varphi(A_1)=\varphi(B_1)+\varphi(C_1)=\varphi(B_2)+\varphi(C_2)=\varphi(A_2).$$

Let $\Ac$ be any unital, diffuse, abelian von Neumann subalgebra of $\Mcal$.
As described above, $E$ is naturally identified with $\Sc(\Ac,\tau\restrict_\Ac)\cap\Ec(\Mcal,\tau)$,
and restricting $\varphi$ to this subalgebra yields a linear functional $\varphi_0$ on $E$, which
is rearrangement-invariant and
satisfies~\eqref{eq:varphidef}, because of the fact that $\varphi(A)$ depends only on $\mu(A)$ for all $A\ge0$.
Using~\eqref{eq:varphidef}, we see that the functional $\varphi_0$ does not depend on $\Ac$, namely, does not depend
on which copy of $E$ we chose in $\Ec(\Mcal,\tau)$.

Finally, as $\varphi$ is uniquely determined by $\varphi_0$ and the condition~\eqref{eq:varphidef}, we see that the map
$\varphi\mapsto\varphi_0$ is the desired bijection.
\end{proof}

For convenience, we will use also $\varphi$, instead of $\varphi_0$, to denote the functional on $E$ corresponding to a trace
$\varphi$ on $\Ec(\Mcal,\tau)$.

For example, taking $E$ to be the function space
$L_1$ of complex-valued functions on $[0,1]$ that are integrable
with respect to Lebesgue measure,
the corresponding bimodule is $\Lc_1(\Mcal,\tau)$.
Moreover, the functional $f\mapsto\int_0^1f(t)\,dt$ on
$L_1$
corresponds to the usual trace $\tau$ on $\Lc_1(\Mcal,\tau)$.
Other examples of traces on bimodules are provided by the Dixmier traces on Marcinkiewicz bimodules, which are of interest in
noncommutative geometry.
See, for example, \cite{DPSS98}, \cite{CS06} and \cite{KSS11};
particularly, consider the treatment of functionals supported at zero, but adapted to the case of a II$_1$-factor $\Mcal$,
namely, corresponding to function spaces on $[0,1]$.
A specific case (essentially, taken from~\cite{DPSS98}) is found in Example~\ref{ex:singular}.

\medskip
The Fuglede-Kadison determinant mentioned at the start of this introduction
is actually naturally defined on the space,
sometimes denoted $\Lc_{\log}(\Mcal,\tau)$,
of all $T\in\Sc(\Mcal,\tau)$ such that $\log_+(|T|)\in\Lc_1(\Mcal,\tau)$,
where $\log_+(t)=\max(\log(t),0)$.
See~\cite{HS07} for a development of $\Delta_\tau$ in this generality, including a proof of multiplicativity.

In the rest of this paper, we will for the most part consider only {\em positive} traces $\varphi$, namely, those satisfying
\[
A\ge0\quad\implies\quad\varphi(A)\ge0
\]
(the exception being Lemma~\ref{commutator lemma}).
Positive traces correspond, under the rubrik of Theorem~\ref{thm:trace}, to positive rearrangement-invariant linear functionals.
In the following, we use the function $\log_-(t)=-\min(\log(t),0)$; thus, $\log=\log_+-\log_-$.
\begin{defi}\label{def:det}
Let $\Mcal$ be a II$_1$-factor and
consider a
positive
trace $\varphi$ on an $\Mcal$-bimodule $\Ec(\Mcal,\tau)$.
Let $\Ec_{\log}(\Mcal,\tau)$ be the set of all
$T\in\Sc(\Mcal,\tau)$ such that $\log_+(|T|)\in\Ec(\Mcal,\tau)$
and for such $T$ let
\[
\tdet_{\varphi}(T)=\begin{cases}
\exp(\varphi(\log(|T|))),&\ker T=\{0\}\text{ and }\log_-(|T|)\in E \\
0,&\ker T=\{0\}\text{ and }\log_-(|T|)\notin E \\
0,&\ker T\ne\{0\}.
\end{cases}
\]
\end{defi}
Thus, in the case $E=L_1$
and $\varphi=\tau$, we have the Fuglede-Kadison determinant:
$\tdet_\tau=\Delta_\tau$.
The natural domain of this determinant by the above rubric should be written $\Lc_{1,\log}(\Mcal,\tau)$,
but we will write $\Lc_{\log}(\Mcal,\tau)$ for this, in keeping with earlier convention ({\em cf}\/ \cite{DSZ.Llog}, \cite{DSZ.unbd}).

The main result of this paper is:
\begin{thm}\label{thm:main}
For an arbitrary Calkin function space $E$ on $[0,1]$ and arbitrary
positive
trace $\varphi$ on the corresponding bimodule
$\Ec(\Mcal,\tau)$,
the set $\Ec_{\log}(\Mcal,\tau)$ is a $*$-subalgebra of $\Sc(\Mcal,\tau)$ and,
if $A,B\in\Ec_{\log}(\Mcal,\tau)$, then
\begin{equation}\label{det mult}
\tdet_{\varphi}(AB)=\tdet_{\varphi}(A)\tdet_{\varphi}(B).
\end{equation}
\end{thm}

The proof, presented in the next section, relies on Fuglede and Kadison's result~\cite{FK52}
that $\Delta_\tau$ is multiplicative on $\Mcal$ and on the characterization from~\cite{DK05}
of sums of $(\Ec(\Mcal,\tau)),\Mcal)$-commutators.
Thus, a special case of this proof yields an alternative proof of Haagerup and Schultz's result~\cite{HS07} about the extension of the
Fuglede--Kadison determinant to $\Lc_{\log}(\Mcal,\tau)$.

\begin{remark}\label{rem:Tinv}
It is immediate
that $\tdet_\varphi(1)=1$ and, for $T\in\Ec_{\log}(\Mcal,\tau)$,
$\tdet_\varphi(T)=0$ if and only if $T$ fails to be invertible in $\Ec_{\log}(\Mcal,\tau)$.
\end{remark}

\begin{remark}\label{rem:phi0}
In the case that $\varphi=0$, we clearly have, for $T\in\Ec_{\log}(\Mcal,\tau)$,
\[
\tdet_\varphi(T)=\begin{cases}
1\text{ if }T\text{ is invertible in }\Ec_{\log}(\Mcal,\tau) \\
0\text{ otherwise.}
\end{cases}
\]
However, if $\varphi\ne0$, then $\tdet_\varphi$ is onto $[0,\infty)$.
\end{remark}

\begin{remark}\label{rem:singular}
It is not difficult to see, in the case $\varphi=\tau$, that Definition~\ref{def:det}
agrees with the definition by equation~\eqref{eq:FK}, in fact even for all $T\in\Lc_{\log}(\Mcal,\tau)$.
However, the analoguous statement is not true for general traces $\varphi$.
In fact, it obviously fails when $\varphi=0$, (see Remark~\ref{rem:phi0}, above).
See Example~\ref{ex:singular} for specific examples of this failure when $\varphi\ne0$.
\end{remark}

We are grateful to Amudhan Krishnaswamy-Usha for asking us a question that led to the next result.
\begin{prop}\label{prop:alles}
For an arbitrary Calkin function space $E$ on $[0,1]$ and an arbitrary map
\[
m:\Ec_{\log}(\Mcal,\tau)\to[0,\infty)
\]
that is multiplicative, order-preserving and nonzero,
there exists a positive trace $\varphi$ on $\Ec(\Mcal,\tau)$
such that $m(X)=\tdet_\varphi(X)$ for every invertible element $X$ in $\Ec_{\log}(\Mcal,\tau)$.
\end{prop}
We will show (in Proposition~\ref{prop:EneF}) that we cannot hope for $m$ to agree with $\tdet_\varphi$ on all of
$\Ec_{\log}(\Mcal,\tau)$.

The proofs of Theorem~\ref{thm:main} and Proposition~\ref{prop:EneF} are contained in the next two sections.

\section{Proof of Theorem~\ref{thm:main}}

Let us begin by describing some further notation and standard conventions.
\begin{enumerate}[label=$\bullet$,leftmargin=20pt]
\item $S(0,1)$ will denote the set of all complex-valued Borel measurable functions on $[0,1]$ and
$L_\infty$
will denote the set of all essentially bounded elements of $S(0,1)$.
As usual, we consider functions that are equal almost everwhere to be the same.
\item
We will apply the Borel functional calculus to self-adjoint elements $T\in\Sc(\Mcal,\tau)$, and will also use
the standard notation 
$T_+=\max(T,0)$ and $T_-=-\min(T,0)$.
\item
For self-adjoint $A\in\Sc(\Mcal,\tau)$, we 
consider its eigenvalue function (or spectral scale), defined for $t\in(0,1)$ by
\[
\lambda(t,A)=\inf\{s\in\Reals\mid \tau(1_{(s,\infty)}(A))\le t\},
\]
where, in accordance with notation for the Borel functional calculus,
$1_{(s,\infty)}(A)$ denotes the spectral projection of $A$ associated to the interval $(s,\infty)$.
This also goes back to Murray and von Neumann.
We will write simply $\lambda(A)$ for the function $t\mapsto\lambda(t,A)$,
which is nonincreasing and right continuous.
Note that, if $A\ge0$, then $\lambda(A)=\mu(A)$.
Moreover, when $a\le b$, with $a\le\lim_{t\to0}\lambda(t,A)$ and $b\ge\lim_{t\to1}\lambda(t,A)$, we have
\begin{align}
\tau\big(A\,1_{[a,b]}(A)\big)&=\int_c^d\lambda(t,A)\,dt, \label{eq:tauA1} \\
\tau\big(1_{[a,b]}(A)\big)&=d-c, \label{eq:tau1}
\end{align}
where
\[
c=\inf\{s\mid \lambda(s,A)\le b\},\qquad d=\sup\{s\mid\lambda(s,A)\ge a\}.
\]
For any $T\in\Sc(\Mcal,\tau)$, since $\mu(T)=\mu(|T|)=\lambda(|T|)$,
from~\eqref{eq:tau1}, we get
\begin{equation}\label{eq:tau1T}
\tau(1_{[0,\mu(t,T)]}(|T|))\ge1-t.
\end{equation}
\item
The following inequalities are standard (see, for example, Corollary 2.3.16 of~\cite{LSZ13}):
for all $A,B\in\Sc(\Mcal,\tau)$, if $s,t>0$ and $s+t<1$, then
\begin{align}
\mu(s+t,A+B)&\le\mu(s,A)+\mu(t,B),\label{eq:muA+B} \\
\mu(s+t,AB)&\le\mu(s,A)\mu(t,B).\label{eq:muAB}
\end{align}
\item
If a function $f$ on $(0,1)$ is right-continuous and monotone, then we will let $\ft$ denote left-continuous version,
namely,
\begin{equation}\label{eq:ft}
\ft(x)=\lim_{t\to x^-}f(t).
\end{equation}
\end{enumerate}

\begin{lemma}\label{main product lemma}
Let $T,S\in\Sc(\Mcal,\tau)$ be self-adjoint.
Then for every $t\in(0,\frac14)$, we have
$$\left|\int_{2t}^{1-2t}\left(\log(\mu(u,e^Te^S))-\lambda(u,T)-\lambda(u,S)\right)\,du\right|\leq 8t\big(\mu(t,T)+\mu(t,S)\big).$$
\end{lemma}
\begin{proof}
Fix $t\in(0,\frac14)$ and, using the continuous functional calculus, set
\begin{align*}
T_0&=\min\{T_+,\mu(t,T)\}-\min\{T_-,\mu(t,T)\}, \\
S_0&=\min\{S_+,\mu(t,S)\}-\min\{S_-,\mu(t,S)\}.
\end{align*}
We have
$$T-T_0=(T_+-\mu(t,T))_+-(T_--\mu(t,T))_+,$$
$$|T-T_0|=(T_+-\mu(t,T))_++(T_--\mu(t,T))_+=(|T|-\mu(t,T))_+.$$
Thus, we have $(T-T_0)1_{[0,\mu(t,T)]}(|T|)=0$ and,
using~\eqref{eq:tau1T}, we get
$\mu(t,T-T_0)=0$;
similarly, we have $\mu(t,S-S_0)=0$.
Using~\eqref{eq:muAB}, for every $u\in(2t,1)$
we have
\begin{alignat*}{2}
\mu(u,e^Te^S)&=\mu(u,e^{T-T_0}\cdot e^{T_0}e^{S_0}\cdot e^{S-S_0})&&\leq\mu(t,e^{T-T_0})\mu(u-2t,e^{T_0}e^{S_0})\mu(t,e^{S-S_0}), \\
\mu(u,e^{T_0}e^{S_0})&=\mu(u,e^{T_0-T}\cdot e^Te^S\cdot e^{S_0-S})&&\leq\mu(t,e^{T_0-T})\mu(u-2t,e^Te^S)\mu(t,e^{S_0-S}),
\end{alignat*}
Since $\mu(t,e^{T-T_0})\leq1$
and $\mu(t,e^{T_0-T})\leq1$
and similarly for $S-S_0$, we get
\[
\mu(u,e^Te^S)\leq\mu(u-2t,e^{T_0}e^{S_0}),
\qquad
\mu(u,e^{T_0}e^{S_0})\leq\mu(u-2t,e^Te^S).
\]
Thus, for $u\in(2t,1-2t)$, we have
$$\mu(u+2t,e^{T_0}e^{S_0})\leq\mu(u,e^Te^S)\leq\mu(u-2t,e^{T_0}e^{S_0}).$$
It follows that
\begin{equation}\label{eq:int2t}
\int_{4t}^1\log(\mu(u,e^{T_0}e^{S_0}))\,du\le\int_{2t}^{1-2t}\log(\mu(u,e^Te^S))\,du\leq\int_0^{1-4t}\log(\mu(u,e^{T_0}e^{S_0}))\,du.
\end{equation}
Since
$-\mu(t,T)\le T_0\le\mu(t,T)$ and similarly for $S_0$,
we also have
$$e^{-\mu(t,T)-\mu(t,S)}\leq\mu(e^{T_0}e^{S_0})\leq e^{\mu(t,T)+\mu(t,S)}.$$
Thus,
$$\|\log(\mu(e^{T_0}e^{S_0}))\|_{\infty}\leq\mu(t,T)+\mu(t,S).$$
In particular,
\begin{gather*}
\left|\int_0^{4t}\log(\mu(u,e^{T_0}e^{S_0}))\,du\right|\leq 4t\|\log(\mu(e^{T_0}e^{S_0}))\|_{\infty}\leq 4t\big(\mu(t,T)+\mu(t,S)\big), \\
\left|\int_{1-4t}^1\log(\mu(u,e^{T_0}e^{S_0}))\,du\right|\leq 4t\|\log(\mu(e^{T_0}e^{S_0}))\|_{\infty}\leq 4t\big(\mu(t,T)+\mu(t,S)\big).
\end{gather*}
Using~\eqref{eq:int2t}, we get
$$\left|\int_{2t}^{1-2t}\log(\mu(u,e^Te^S))\,du-\int_0^1\log(\mu(u,e^{T_0}e^{S_0}))\,du\right|\leq 4t\big(\mu(t,T)+\mu(t,S)\big).$$
Since the Fuglede-Kadison determinant $\Delta_\tau$ is multiplicative on $\Mcal$, we have
\begin{align*}
\int_0^1\log(\mu(u,e^{T_0}e^{S_0}))\,du&=\log(\Delta_\tau(e^{T_0}e^{S_0})) \\
&=\log(\Delta_\tau(e^{T_0}))+\log(\Delta_\tau(e^{S_0}))=\tau(T_0)+\tau(S_0).
\end{align*}
But using
$$\left|\tau(T_0)-\int_{2t}^{1-2t}\lambda(u,T)\,du\right|\le 4t\mu(t,T),$$
and the same also for $S$, the assertion follows.
\end{proof}

In the following, we use the notation~\eqref{eq:ft} for the left-continuous versions of monotone functions.
(Though, as elements of $E$, $\mu(T)$ and the left-continuous version $\mut(T)$ are identified, 
these functions $\mu(T)$ and similarly $\lambda(T)$ are of interest aside from their membership in $E$, and for correctness
at all points of $(0,1)$ we must use their left-continuous versions in the following inequalities and elsewhere below.)
\begin{lemma}\label{product is good}
If $S,T\in\Sc(\Mcal,\tau)$ are self-adjoint, then for all $u\in(0,1)$, we have
\begin{equation}\label{eq:mueTeS}
-\mut(\frac{1-u}{2},T)-\mut(\frac{1-u}{2},S)\leq\log(\mu(u,e^Te^S))\leq\mu(\frac{u}{2},T)+\mu(\frac{u}{2},S).
\end{equation}
\end{lemma}
\begin{proof} 
Using~\eqref{eq:muAB},
we get
\begin{align}
\mu(u,e^Te^S)\leq \mu(\frac{u}{2},e^T)\mu(\frac{u}{2},e^S)\leq\mu(\frac{u}{2},e^{T_+})\mu(\frac{u}{2},e^{S_+})
&=e^{\mu(\frac{u}{2},T_+)+\mu(\frac{u}{2},S_+)} \label{eq:topline}\\
&\leq e^{\mu(\frac{u}{2},T)+\mu(\frac{u}{2},S)}, \notag
\end{align}
which yields the right-most inequality in~\eqref{eq:mueTeS}.
Replacing $S$ with $-T$ and $T$ with $-S$ in~\eqref{eq:topline}, we get
\begin{equation}\label{eq:muu}
\mu(u,e^{-S}e^{-T})\leq e^{\mu(\frac{u}{2},T_-)+\mu(\frac{u}{2},S_-)},\qquad
\mut(u,e^{-S}e^{-T})\leq e^{\mut(\frac{u}{2},T_-)+\mut(\frac{u}{2},S_-)}.
\end{equation}
As is well known and easy to show,
$$\mu(u,e^Te^S)=\frac1{\mut(1-u,e^{-S}e^{-T})}.$$
Thus, replacing $u$ with $1-u$ in~\eqref{eq:muu}, we get
$$\mu(u,e^Te^S)\geq e^{-\mut(\frac{1-u}{2},T_-)-\mut(\frac{1-u}{2},S_-)}\ge e^{-\mut(\frac{1-u}{2},T)-\mut(\frac{1-u}{2},S)},$$
which yields the left-most inequality in~\eqref{eq:mueTeS}.
\end{proof}

The next lemma is a combination of Theorems 3.3.3 and 3.3.4 from~\cite{LSZ13}.
\begin{lemma}\label{majorization lem} If $S,T\in\mathcal{M}$ are positive, then
$$\int_0^t\mu(u,T+S)\,du\leq\int_0^t\big(\mu(u,T)+\mu(u,S)\big)\,du\leq\int_0^{2t}\mu(u,T+S)\,du.$$
\end{lemma}
\begin{proof}
This follows easily from the fact that, for a positive operator, $T$, we have
\[
\int_0^t\mu(u,T)\,du=\sup\{\tau(pT)\mid p\in\Proj(\Mcal),\,\tau(p)\le t\}.
\]
\end{proof}

For every function $f\in S(0,1)$ that is bounded on compact subsets of $(0,1)$, define
\[
(\Psi f)(t)=\begin{cases}
\frac1t\int_t^{1-t}f(s)\,ds,&0<t<\frac12, \\[1ex]
0,&\frac12\le t\le 1.
\end{cases}
\]
Clearly,
$\Psi f$ is continuous on $(0,1]$ and $\Psi$ is linear.
Note that $\Psi$ is defined on every function arising as $\mu(A)$ or $\lambda(A)$ for $A\in\Sc(\Mcal,\tau)$.

\begin{lemma}\label{sum pos lemma} Let $S,T\in\Ec(\Mcal,\tau)$ be positive.
Then
$$\Psi(\mu(T+S)-\mu(T)-\mu(S))\in E.$$
\end{lemma}
\begin{proof}
First suppose $S,T\in\mathcal{M}$ are positive.
{}From Lemma \ref{majorization lem} and the fact that $\tau(T)=\int_0^1\mu(u,T)\,du$, we have
\begin{equation}\label{hvost}
\int_{2t}^1\mu(u,T+S)\,du\leq\int_t^1\big(\mu(u,T)+\mu(u,S)\big)\,du\leq\int_t^1\mu(u,T+S)\,du.
\end{equation}
For arbitrary positive $S,T\in\Sc(\Mcal,\tau)$, set $T_n=\min\{T,n\}$ and $S_n=\min\{S,n\}$.
Since $\mu(T_n)\uparrow\mu(T)$, $\mu(S_n)\uparrow\mu(S)$ and $\mu(T_n+S_n)\uparrow\mu(T+S)$,
it follows from the Monotone Convergence Principle that \eqref{hvost} also holds.
{}From \eqref{hvost}, we have
\[
\left|\int_t^1\big(\mu(u,T+S)-\mu(u,T)-\mu(u,S)\big)\,du\right|\leq
\int_t^{2t}\mu(u,T+S)\,du\le  t\mu(t,T+S).
\]
Thus, for $t\in(0,\frac12),$ we have
\begin{align*}
\bigg|\int_t^{1-t}&\big(\mu(u,T+S)-\mu(u,T)-\mu(u,S)\big)\,du\,\bigg| \\
&\le\bigg|\int_t^1\big(\mu(u,T+S)-\mu(u,T)-\mu(u,S)\big)\,du\,\bigg| \\
&\qquad\qquad\qquad\qquad+\bigg|\int_{1-t}^1\big(\mu(u,T+S)-\mu(u,T)-\mu(u,S)\big)\,du\,\bigg| \\[1ex]
&\leq t\mu(t,T+S)+t\mu(1-t,T+S)+t\mu(1-t,T)+t\mu(1-t,S)\leq 4t\mu(t,T+S).
\end{align*}
This concludes the proof.
\end{proof}

\begin{lemma}\label{tpm lemma} Let $T\in\Sc(\Mcal,\tau)$ be self-adjoint.
Then
$$\Psi(\lambda(T)-\mu(T_+)+\mu(T_-))\in L_{\infty}.$$
\end{lemma}
\begin{proof}
If $T_+=0$ or $T_-=0$, then $\lambda(T)=\mu(T_+)-\mu(T_-)$.
Suppose $T_+\neq0$ and $T_-\neq0$.
Let $t_0$ be the trace of the support projection of $T_+$.
We have
\[
\lambda(u,T)=\begin{cases}
\mu(u,T_+),&u\in(0,t_0) \\
-\mut(1-u,T_-),&u\in[t_0,1).
\end{cases}
\]
It follows that, for all sufficiently small $t$, we have
\begin{multline*}
t(\Psi\lambda(T))(t)=\int_t^{t_0}\lambda(u,T)\,du+\int_{t_0}^{1-t}\lambda(u,T)\,du \\
=\int_t^{t_0}\mu(u,T_+)\,du-\int_{t_0}^{1-t}\mu(1-u,T_-)\,du=\int_t^{t_0}\mu(u,T_+)\,du-\int_t^{1-t_0}\mu(u,T_-)\,du \\
=\int_t^1\big(\mu(u,T_+)-\mu(u,T_-)\big)\,du=t(\Psi(\mu(T_+)-\mu(T_-)))(t),
\end{multline*}
where the last equality holds because the integrand is zero when $u$ is sufficiently close to $1$.
Thus, $\Psi(\lambda(T)-\mu(T_+)+\mu(T_-))(t)$ vanishes for all $t$ sufficiently small.
Since this function is continuous on $(0,1]$, it
is bounded.
\end{proof}

\begin{lemma}\label{sum lemma} Let $S,T\in\Ec(\Mcal,\tau)$ be self-adjoint.
Then
$$\Psi(\lambda(T)+\lambda(S)-\lambda(T+S))\in E.$$
\end{lemma}
\begin{proof} We have
$$(T+S)_+-(T+S)_-=T_+-T_-+S_+-S_-.$$
Therefore,
$$(T+S)_++T_-+S_-=(T+S)_-+T_++S_+.$$
Denote the above quantity by $A$.
{}From Lemma \ref{sum pos lemma}, we obtain
\begin{gather*}
\Psi(\mu(A)-\mu((T+S)_+)-\mu(T_-)-\mu(S_-))\in E, \\
\Psi(\mu(A)-\mu((T+S)_-)-\mu(T_+)-\mu(S_+))\in E.
\end{gather*}
Subtracting those formulae, we obtain
$$\Psi(\mu((T+S)_+)-\mu((T+S)_-)-\mu(T_+)+\mu(T_-)-\mu(S_+)+\mu(S_-))\in E.$$
The assertion follows now from Lemma \ref{tpm lemma} as applied to the operators $T$, $S$ and $T+S$,
and the fact that $E$ contains $L_\infty$.
\end{proof}

In the next result, the notation $[\Ec(\Mcal,\tau),\Mcal]$ denotes the space spanned by the set of
all commutators of the form $[S,T]=ST-TS$,
for $S\in\Mcal$ and $T\in\Ec(\Mcal,\tau)$.
It amounts to a reformulation of a special case of Theorem~4.6 of~\cite{DK05}.
\begin{thm}\label{commutator thm}
Let $T\in\Ec(\Mcal,\tau)$ be self-adjoint.
Then $T\in[\Ec(\Mcal,\tau),\Mcal]$ if and only if $\Psi\lambda(T)\in E$. 
\end{thm}
\begin{proof}
By
Theorem~4.6 of~\cite{DK05}, $T\in[\Ec(\Mcal,\tau),\Mcal]$ if and only if the function
\[
r\mapsto \frac1r\tau(1_{[0,\mu(r,T)]}(|T|)T)
\]
belongs to $E$.
Thus, it will suffice to show that the function
\begin{equation}\label{eq:rfunc}
r\mapsto\frac1r\tau(1_{[0,\mu(r,T)]}(|T|)T)-\Psi\lambda(T)(r)
\end{equation}
belongs to $E$.
First suppose $T_-=0$.
Then, using $\lambda(T)=\mu(T)$ and~\eqref{eq:tauA1}, we have
\[
\tau(1_{[0,\mu(r,T)]}(|T|)T)=\int_{r'}^1\mu(t,T)\,dt,
\]
where $r'=\inf\{s\mid \mu(s,T)\le\mu(r,T)\}$.
Thus $r'\le r$ and, for $0<r<\frac12$,
\[
\left|\tau(1_{[0,\mu(r,T)]}(|T|)T)-\int_r^{1-r}\lambda(t,T)\,dt\right|\le(r-r')\mu(r,T)+r\mu(1-r,T)\le2r\mu(r,T) ,
\]
which implies that the function~\eqref{eq:rfunc} belongs to $E$.

If $T_+=0$, then we may of course replace $T$ by $-T$ and we are done.

Suppose $T_+\ne0$ and $T_-\ne0$.
Letting, $t_0=\inf\{t\mid\lambda(t,T_+)\ge0\}$, we have $0<t_0<1$ and
\[
\lambda(t,T)=\begin{cases}
\mu(t,T_+),&0<t<t_0 \\
\displaystyle
\mut(1-t,T_-),&t_0\le t<1.
\end{cases}
\]
For $r<t_0$, we have
\begin{multline*}
\tau(1_{[0,\mu(r,T)]}(|T|)T)=\tau(1_{[-\mu(r,T),\mu(r,T)]}(T)T)
=\tau(1_{[0,\mu(r,T)]}(T_+)T_+)-\tau(1_{[0,\mu(r,T)]}(T_-)T_-) \\
=\int_{r'}^{t_0}\lambda(t,T)\,dt+\int_{t_0}^{1-r''}\lambda(t,T)\,dt,
\end{multline*}
where
\begin{align}
r'&=\inf\{s\mid\mu(s,T_+)\le\mu(r,T)\} \label{eq:r'} \\
r''&=\inf\{s\mid\mu(s,T_-)\le\mu(r,T)\}. \label{eq:r''}
\end{align}
Since $\mu(r,T_\pm)\le\mu(r,T)$, we have $r',r''\le r$.
Thus, we have
\begin{multline*}
\left|\tau(1_{[0,\mu(r,T)]}(|T|)T)-\int_r^{1-r}\lambda(t,T)\,dt\right|
=\left|\int_{r'}^r\lambda(t,T)\,dt+\int_{1-r}^{1-r''}\lambda(t,T)\,dt\right| \\
\le\int_{r'}^r\mu(t,T_+)\,dt+\int_{r''}^r\mu(t,T_-)\,dt
\le(r-r')\mu(r',T_+)+(r-r'')\mu(r'',T_-)\le2r\mu(r,T),
\end{multline*}
where for the last inequality we used \eqref{eq:r'}--\eqref{eq:r''}.
This shows that the function~\eqref{eq:rfunc} belongs to $E$ and, thus,
completes the proof.
\end{proof}

\begin{lemma}\label{commutator lemma}
Let $\varphi:\Ec(\Mcal,\tau)\to\mathbb{C}$ be a trace.
If $T\in\Ec(\Mcal,\tau)$ is
self-adjoint and is
such that $\Psi\lambda(T)\in E,$ then $\varphi(T)=0.$
\end{lemma}
\begin{proof}
It follows from Theorem \ref{commutator thm} that $T\in[\Ec(\Mcal,\tau),\Mcal]$.
Since $\varphi$ is a trace, it follows that $\varphi(T)=0.$  
\end{proof}

\begin{proof}[Proof of Theorem~\ref{thm:main}]
For $A\in\Sc(\Mcal,\tau)$, we have that $A\in\Ec_{\log}(\Mcal,\tau)$ if and only if $\log_+\mu(A)\in E$,
and this is, in turn, equivalent to $\log(1+\mu(A))\in E$.
Using the basic equalities \eqref{eq:muA+B}-\eqref{eq:muAB}, we easily see that
for $A,B\in\Ec_{\log}(\Mcal,\tau)$, we have
\begin{gather*}
\log(1+\mu(A+B))\le\log(1+D_2\mu(A)+D_2\mu(B))\le\log\big((1+D_2\mu(A))(1+D_2\mu(B))\big) \\
\log(1+\mu(AB))\le\log(1+D_2\mu(A)D_2\mu(B))\le\log\big((1+D_2\mu(A))(1+D_2\mu(B))\big),
\end{gather*}
where $(D_2f)(t)=f(t/2)$.
But since
$\log(1+D_2\mu(A))+\log(1+D_2\mu(B))\in E$, these
imply that $A+B$ and $AB$ belong to $\Ec_{\log}(\Mcal,\tau)$.
{}From this, one easily sees that $\Ec_{\log}(\Mcal,\tau)$ is a $*$-subalgebra of $\Sc(\Mcal,\tau)$.

It remains to show that $\tdet_\varphi$ is multiplicative.
Letting $A,B\in\Ec_{\log}(\Mcal,\tau)$, we will show~\eqref{det mult}.
We may, without loss of generality, assume $A,B\ge0$.
Indeed, we have $\mu(AB)=\mu(|A| |B^*|)$.
Thus, if the assertion
holds for positive operators, then we will have
\[
\tdet_{\varphi}(AB)=\tdet_{\varphi}(|A||B^*|)=\tdet_{\varphi}(|A|)\tdet_{\varphi}(|B^*|)=\tdet_{\varphi}(A)\tdet_{\varphi}(B).
\]
Suppose first that $\log(A),\log(B)\in\Ec(\Mcal,\tau)$.
Denote, for brevity, $T=\log(A)$ and $S=\log(B)$.
It follows from Lemma \ref{product is good} that $\log(|AB|)\in E$.

Using Lemma \ref{main product lemma} and replacing $t$ with $\frac12t$, for all $t\in(0,\frac12)$, we get
$$\Big|\int_t^{1-t}\big(\log(\mu(u,e^Te^S))-\lambda(u,T)-\lambda(u,S)\big)\,du\Big|\leq 4t\big(\mu(\frac{t}{2},T)+\mu(\frac{t}{2},S)\big).$$
In particular, we have
$$\Psi(\log(\mu(e^Te^S))-\lambda(T)-\lambda(S))\in E.$$
It follows from Lemma \ref{sum lemma} that
$$\Psi\Big(\lambda\big(\log(|e^Te^S|)-T-S\big)\Big)\in E.$$
Using Lemma \ref{commutator lemma}, we conclude that
$$\varphi(\log(|e^Te^S|)-T-S)=0.$$
This implies \eqref{det mult} for our $A,B$.

If $B$ has a nonzero kernel, then so does $AB$ and~\eqref{det mult}  holds.

Suppose now that $\ker B$ is zero but $\log_-(B)\notin E$.
Then, of course, $\lim_{t\to1}\mu(t,B)=0$.
If $\ker AB\ne\{0\}$, then~\eqref{det mult} holds, so suppose $\ker AB=\{0\}$.
We have, from~\eqref{eq:muAB}, for all $t\in(0,\frac12)$,
$$\mu(1-t,AB)\leq\mu(t,A)\mu(1-2t,B)$$
and, thus,
$$\log(\mu(1-t,AB))\leq\log(\mu(t,A))+\log(\mu(1-2t,B)).$$
So, for sufficiently small $t>0$,
\begin{multline*}
\log_-\mu(1-t,AB)+\log_+\mu(t,A)\ge-\log\mu(1-t,AB)+\log\mu(t,A) \\
\ge-\log\mu(1-2t,B)=\log_-\mu(1-2t,B).
\end{multline*}
Since the function $t\mapsto\log_-\mu(1-2t,B)$ is not in $E$, while the function $t\mapsto\log_+\mu(t,A)$ does belong to $E$,
we conclude that the function $t\mapsto\log_-\mu(1-t,AB)$ does not belong to $E$.
Therefore, the function $\log_-(\mu(AB))$ does not belong to $E$ and both left- and right-hand sides of \eqref{det mult} are zero.
This concludes the proof of~\eqref{det mult} in the degenerate case.
\end{proof}

\section{Proof of Proposition~\ref{prop:alles} and some examples}

\begin{lemma}\label{lem:mmu} Let $m:\Ec_{\log}(\mathcal{M},\tau)\to\mathbb{R}$ be multiplicative and order-preserving.
Then for every $T\in\Ec_{\log}(\mathcal{M},\tau)$, $m(T)$ depends only on $\mu(T).$
\end{lemma}
\begin{proof}
We may without loss of generality assume $m$ is not identically zero.
Thus, $m(1)=1$.
By Theorem 1 of~\cite{Br67}, every unitary element is a product of multiplicative commutators of unitaries (in fact, of symmetries)
and it follows that $m$ sends the entire unitary group of $\Mcal$ to $1.$ Thus, by employing the polar decomposition, we have
\[ 
\forall T\in\Ec_{\log}(\Mcal,\tau),\quad m(T)=m(|T|).
\] 
It, therefore, suffices to prove the assertion for positive operators.

Let $0\leq T,S\in\Ec_{\log}(\mathcal{M},\tau)$ be such that $\mu(T)=\mu(S).$ Set
$$T_{\epsilon}=\sum_{n\in\mathbb{Z}}(1+\epsilon)^n1_{((1+\epsilon)^n,(1+\epsilon)^{n+1})}(T),\quad S_{\epsilon}=\sum_{n\in\mathbb{Z}}(1+\epsilon)^n1_{((1+\epsilon)^n,(1+\epsilon)^{n+1})}(S).$$
For a given $n,$ positive operators $T_{\epsilon}$ and $S_{\epsilon}$ have discrete spectrum and $\mu(T_{\epsilon})=\mu(S_{\epsilon}).$ Since $\Mcal$ is a factor, one can choose a unitary operator $U_{\epsilon}\in\Mcal$ such that $S_{\epsilon}=U_{\epsilon}T_{\epsilon}U_{\epsilon}^{-1}.$ Thus,
$$m(S_{\epsilon})=m(U_{\epsilon}T_{\epsilon}U_{\epsilon}^{-1})=m(U_{\epsilon})m(T_{\epsilon})m(U_{\epsilon})^{-1}=m(T_{\epsilon}).$$
Clearly,
$$S_{\epsilon}\leq S\leq (1+\epsilon)S_{\epsilon},\quad T_{\epsilon}\leq T\leq (1+\epsilon)T_{\epsilon}.$$
Since $m$ is order preserving, it follows that
$$m(S)\leq m(1+\epsilon)m(S_{\epsilon})=m(1+\epsilon)m(T_{\epsilon})\leq m(1+\epsilon)m(T).$$
Since $m$ is order preserving, it follows that $m(1+\epsilon)\searrow1$ as $\epsilon\searrow0.$ Passing $\epsilon\to0,$ we obtain $m(S)\leq m(T)$.
Similarly, $m(T)\leq m(S).$ Thus, $m(S)=m(T)$ and the proof is complete.
\end{proof}

\begin{proof}[Proof of Proposition~\ref{prop:alles}] Since the map $m$ is multiplicative and not identically zero, we must have $m(1)=1$.
By Lemma~\ref{lem:mmu}, $m(T)$ depends only on $\mu(T)$ for all $T\in\Ec_{\log}(\Mcal,\tau)$.

Let $\Ac$ be any unital, diffuse, abelian von Neumann subalgebra of $\Mcal$.
As in the proof of Theorem~\ref{thm:main}, $E$ is naturally identified with $\Sc(\Ac,\tau\restrict_\Ac)\cap\Ec(\Mcal,\tau)$.
Given real-valued $f\in E$, let $T\in\Sc(\Ac,\tau\restrict_\Ac)\cap\Ec(\Mcal,\tau)$ be the corresponding self-adjoint operator.
Note that $e^T$ is an invertible element of $\Ec_{\log}(\Mcal,\tau)$ and, thus, $m(e^T)>0$.
We define
\begin{equation}\label{eq:logmeT}
\varphi_0(f)=\log m(e^T).
\end{equation}
We will show that $\varphi_0$ is $\Reals$-linear.
First, given $f_1,f_2\in E$ and the corresponding self-adjoint $T_1,T_2\in\Sc(\Ac,\tau\restrict_\Ac)\cap\Ec(\Mcal,\tau)$, since $T_1$ and $T_2$ commute,
we have
\[
\varphi_0(f_1+f_2)=\log m(e^{T_1+T_2})=\log m(e^{T_1}e^{T_2})=\log\big(m(e^{T_1})m(e^{T_2})\big)=\varphi_0(f_1)+\varphi_0(f_2),
\]
i.e., $\varphi_0$ preserves addition.
From this, we easily see that $\varphi_0(rf)=r\varphi_0(f)$ for every rational number $r$ and real-valued $f\in E$.
This last fact is, of course, equivalent to 
\begin{equation}\label{eq:merT}
m(e^{rT})=m(e^T)^r
\end{equation}
for every self-adjoint $T\in\Sc(\Ac,\tau\restrict_\Ac)\cap\Ec(\Mcal,\tau)$ and every rational number $r$.
When $T\ge0$, using the order-preserving property of $m$, we obtain from this that~\eqref{eq:merT} holds for every $r\in\Reals$,
and similarly when $T\le0$.
For arbitrary self-adjoint $T\in\Sc(\Ac,\tau\restrict_\Ac)\cap\Ec(\Mcal,\tau)$, writing $T=T_+-T_-$ for $T_+$ and $T_-$ positive
elements of $\Sc(\Ac,\tau\restrict_\Ac)\cap\Ec(\Mcal,\tau)$, in the usual way, we get, for all $r\in\Reals$,
\begin{multline*}
m(e^{rT})=m(e^{rT_++(-r)T_-})=m(e^{rT_+})m(e^{(-r)T_-})=m(e^{T_+})^rm(e^{T_-})^{-r} \\
=\big(m(e^{T_+}e^{-T_-})\big)^r=m(e^T)^r.
\end{multline*}
Thus~\eqref{eq:merT} holds for all self-adjoint $T$ and all $r\in\Reals$,
and it follows that $\varphi_0(rf)=r\varphi_0(f)$ for all real-valued $f\in E$ and all $r\in\Reals$.
Thus, we have defined an $\Reals$-linear functional $\varphi_0$ on the space of real-valued elements of $E$.
Complexification extends $\varphi_0$ to a $\Cpx$-linear functional on $E$.

We now observe that $\varphi_0$ is rearrangement-invariant.
If $f\in E$ and $f\ge0$ and if $T\in\Sc(\Ac,\tau\restrict_\Ac)\cap\Ec(\Mcal,\tau)$ is the corresponding element,
then  $\mu(e^T)=e^{f^*}$, where $f^*$ is the nondecreasing rearrangement of $f$.
Since $m(e^T)$ depends only on $\mu(e^T)$,
we see that $\varphi_0(f)=\varphi_0(f^*)$ and, thus, $\varphi_0$ is rearrangement-invariant.

By Theorem~\ref{thm:trace}, there is a unique trace $\varphi$ on $\Ec(\Mcal,\tau)$ such that $\varphi(T)=\varphi_0(\mu(T))$
whenever $T\in\Ec(\Mcal,\tau)$ is positive.
Suppose $X$ is an invertible element of $\Ec_{\log}(\Mcal,\tau)$ and let us observe that $m(X)=\tdet_\varphi(X)$.
Since $m(X)=m(|X|)$ and likewise for $\tdet_\varphi$, we may without loss of generality assume $X\ge0$.
Thus, there is self-adjoint $T=\log(X)\in\Ec(\Mcal,\tau)$ such that $X=e^T$.
Thus, by~\eqref{eq:logmeT}, we have
\[
m(X)=e^{\varphi_0(\lambda(T))}=e^{\varphi(T)}=\tdet_\varphi(X),
\]
as required.
\end{proof}

The following shows that Proposition~\ref{prop:alles} cannot be improved to obtain $m=\tdet_\varphi$ on all of $\Ec_{\log}(\Mcal,\tau)$.
\begin{prop}\label{prop:EneF}
Let $E$ be a symmetric function space. Consider strictly larger symmetric function space $F.$ If $\psi$ is an arbitrary positive trace on $\mathcal{F}(\Mcal,\tau),$ then
$$\tdet_\psi\restrict_{\Ec_{\log}(\Mcal,\tau)}\neq \tdet_{\varphi}$$
for each positive trace $\varphi$ on $\Ec(\Mcal,\tau).$
\end{prop}
\begin{proof} To see this, fix $0\leq T\in \mathcal{F}(\Mcal,\tau)$ such that $T\notin \mathcal{E}(\Mcal,\tau).$ Take $X=e^{-T}$.
Then $X$ is bounded, so belongs to $\Ec_{\log}(\Mcal,\tau)$.
Moreover, $X^{-1}=e^T$ belongs to $\Fc_{\log}(\Mcal,\tau)$, but $X$ is not invertible in $\Ec_{\log}(\Mcal,\tau)$.
Thus, we have
$$\tdet_{\psi}(X)=e^{-\psi(T)}\ne0=\tdet_{\varphi}(X).$$
\end{proof}

See Remark~\ref{rem:singular} for the relevance of the following example.

\begin{example}\label{ex:singular}
We give examples of a nonzero trace $\varphi$ on a bimodule $\Ec(\Mcal,\tau)$ and $T\in\Ec(\Mcal,\tau)$ such that $\varphi\ne0$
but
\begin{equation}\label{eq:notlim}
\tdet_\varphi(T)\ne\lim_{\eps\to0^+}\tdet_\varphi(|T|+\eps).
\end{equation}
Let $\psi$ be an increasing, continuous, concave function on the interval $[0,1]$ satisfying
\[
\lim_{t\to0}\frac{\psi(2t)}{\psi(t)}=1.
\]
For example, take $\psi(t)=\frac1{2-\log(t)}$.
Let $E=M_\psi$ be the Marcinkiewicz space
\[
E=\left\{f\in S(0,1)\;\biggm|\;\sup_{0<t<1}\frac1{\psi(t)}\int_0^tf^*(s)\,ds<\infty\right\},
\]
where $f^*$ is the decreasing rearrangement of $|f|$.
Let $\Ec(\Mcal,\tau)$ be the corresponding $\Mcal$-bimodule.
By Example~2.5(ii) of~\cite{DPSS98}, there is a positive, rearrangement-invariant,
linear functional $\varphi$ on $E$ that vanishes on $E\cap L_\infty$,
but satisfies $\varphi(\psi')=1$.
For $f\in E$ with $f\ge0$, $\varphi(f)$ is realized as a particular sort of generalized limit as $t\to0$
of $\frac1{\psi(t)}\int_0^tf^*(s)\,ds$.
Let $\varphi$ denote also the trace on $\Ec(\Mcal,\tau)$, according to Theorem~\ref{thm:trace}.
Thus, we have $\tdet_\varphi(T)=1$ whenever $T\in\Mcal$ is bounded and has bounded inverse.
Consequently, if $T\in\Mcal$ fails to be invertible in $\Ec_{\log}(\Mcal,\tau)$, for example, because it has a nonzero kernel,
then, by Definition \ref{def:det}, $\tdet_\varphi(T)=0$, but the right-hand-side of~\eqref{eq:notlim} is equal to $1$.

The examples considered hitherto involved non-invertible elements of $\Ec_{\log}(\Mcal,\tau)$.
However,~\eqref{eq:notlim} can also fail when $T$ is invertible in $\Ec_{\log}(\Mcal,\tau)$.
For example, take $T\ge0$ such that $\mu(T)(t)=\exp(-\psi'(1-t))$.
In particular, $T$ is bounded.
Then $\tdet_\varphi(T)=e^{-1}$ but again the right-hand-side of~\eqref{eq:notlim} is equal to $1$.
\end{example}

\medskip
\noindent
{\bf Acknowldegement:}
The authors thank Amudhan Krishnaswamy-Usha for valuable discussions.

\end{document}